\documentclass[11pt,twoside]{amsart}
\usepackage{amsmath, amsthm, amscd, amsfonts, amssymb, graphicx, color}
\usepackage[bookmarksnumbered, colorlinks, plainpages]{hyperref}

\newtheorem{theorem}{Theorem}[section]

\newtheorem{corollary}[theorem]{Corollary}

\theoremstyle{definition}
\newtheorem{definition}[theorem]{Definition}
\theoremstyle{remark}

\newtheorem{example}[theorem]{Example}

\newtheorem{note}[theorem]{Note}

\numberwithin{equation}{section}
\def\pn{\par\noindent}

\begin{document}


\title{Relations on continuities and boundedness in intuitionistic fuzzy pseudo normed linear spaces}
\author{Bivas Dinda, Santanu Kumar Ghosh, T.K. Samanta}


\maketitle

\begin{abstract}
In this study different types of intuitionistic fuzzy continuities (IFCs) and intuitionistic fuzzy boundedness (IFBs) in intuitionistic fuzzy pseudo normed linear spaces are studied. Relations (intra and inter) on intuitionistic fuzzy continuities and intuitionistic fuzzy boundedness are investigated.
\end{abstract}

\vskip 0.2 true cm

{\scriptsize
\hskip -0.4 true cm {\bf Key Words:} Strongly intuitionistic fuzzy continuity, weakly intuitionistic fuzzy continuity, sequentially intuitionistic fuzzy continuity, strongly intuitionistic fuzzy bounded, weakly intuitionistic fuzzy bounded, uniformly intuitionistic fuzzy bounded.}

\pagestyle{myheadings}
\markboth{\rightline {\scriptsize  Bivas Dinda, S.K. Ghosh and T.K. Samanta}}
         {\leftline{\scriptsize Relations on continuities and boundedness in IFPNLS}}

\bigskip
\section{Introduction}
The fuzzy norm concept was originated by A. Katsaras \cite{katsaras1, katsaras2}. Subsequently, this notion swotted by multiple researchers, viz. C. Falbin \cite{felbin}, S.C. Cheng and J. N. Moderson \cite{cheng}, I. Golet \cite{golet} and many others. Chasing the conviction of Cheng-Moderson\cite{cheng}, Bag-Samanta\cite{bag1} considered another definition of fuzzy norm, it became most acceptable among researchers. Motivated by the work of Bag-Samanta \cite{bag1, bag2}, S. N\u ad\u aban \cite{nadaban} introduce the idea of fuzzy pseudo norm.\\
Study of intuitionistic fuzzy normed spaces attracted lots of interest in recent years \cite{dinda1, dinda2, dinda3, dinda4, saadati, samanta}. In particular, Dinda et al. \cite{dinda1} swotted the concept of intuitionistic fuzzy pseudo normed linear spaces and deduced that the concept of of intuitionistic fuzzy pseudo norm is more general set up than intuitionistic fuzzy norm.\\
In this paper, intuitionistic fuzzy continuities and intuitionistic fuzzy boundedness of linear operator are studied in intuitionistic fuzzy pseudo normed spaces, a more general concept of intuitionistic fuzzy normed spaces. In section 3, the concept of intuitionistic fuzzy continuities and intra relation on various types of intuitionistic fuzzy continuities are emphasized. In section 4,various types of intuitionistic fuzzy boundedness are studied. First intra relations on different types of intuitionistic fuzzy boundedness is obtained, then the interrelations on different types of continuities and boundedness are gone into.

\section{Preliminaries}
\begin{definition}\cite{tvsb1}
Let $X$ be a linear space. A function $\|\cdot\|\,:\,X\to\mathbb{R}$ is said to be a pseudo norm on $X$ if the following conditions hold:\\
$(P.1)\, \|x\|\,\geq\,0,\;\forall\,x\in X$;\\
$(P.2)\, \|x\|\,=\,0$ if and only if $x=\theta,\;\forall\,x\in X$;\\ 
$(P.3)\, \|c\,x\|\,\leq\,\|x\|,\;\forall\,x\in X,\;\forall\,c\in\,\mathbb{K}$ with $\,|c|\,\leq\,1$;\\
$(P.4)\, \|x+y\|\,\leq\,\|x\|+\|y\|,\;\forall\,x,y\in X$.
\end{definition}

\begin{definition}\cite{dinda1}\label{d1}
Let $X$ be linear space over the field $\mathbb{K}\,(=\,\mathbb{R}\,/\,\mathbb{C}\,)$. A intuitionistic fuzzy subset $(\mu,\nu)$ of $(X\times\mathbb{R}\,X\times\mathbb{R})$ is said to be an IFPN on $X$ if 
$\forall\,x,y\in\,X$ \\
\textbf{(IFP.1)}$\,\,\mu(\,x\,,\,t\,)\,+\,\nu(\,x\,,\,t\,)\,\leq\,1.$\\
\textbf{(IFP.2)} $\,\,\forall t\,\in \,\mathbb{R}$ with $t\,\leq\,0\,,\,\mu(\,x\,,\,t\,)\,=\,0\,$;\\
\textbf{(IFP.3)} $\,\,\forall t\,\in \,\mathbb{R}^+,\,\mu(\,x\,,\,t\,)\,=\,1\,$ if and only if $x\,=\,\theta$;\\
\textbf{(IFP.4)} $\,\forall t\,\in \,\mathbb{R}^+,\,\mu\,(\,cx\,,\,t\,)\,\geq\,\mu\,(\,x\,,\,t\,)\,$ if $\mid c \mid\,\leq\,1,\,\forall\,c\,\in \mathbb{K}$;\\
\textbf{(IFP.5)}  $\,\mu\,(\,x\,+\,y\,,\,s\,+\,t\,)\,\geq\,\min\,\{\mu\,(\,x\,,\,s\,)\,,\,\mu\,(\,y\,,\,t\,)\,\},\,\,\forall s,\,t\,\in\,\mathbb{R}^+$;\\ 
\textbf{(IFP.6)}  $\mathop{\lim }\limits_{t\, \to\,\,\infty }\,\mu\,(\,x\,,\,t\,)=1.$\\
\textbf{(IFP.7)} if there exists $\,\alpha\,\in\,(0,1)\,$ such that $\,\mu(x\,,\,t)\,>\,\alpha\,,\,\forall\,t\,\in\,\mathbb{R^+}\,$  then $x=\theta$.\\
\textbf{(IFP.8)} $\,\forall\,x\,\in\,X,\,\, \mu(\,x\,,\,\cdot)\,$ is left continuous on $\,\mathbb{R}\,$.\\
\textbf{(IFP.9)} $\,\,\forall t\,\in \,\mathbb{R}$ with $t\,\leq\,0\,,\,\nu(\,x\,,\,t\,)\,=\,1\,$;\\
\textbf{(IFP.10)} $\,\,\forall t\,\in \,\mathbb{R}^+,\,\nu(\,x\,,\,t\,)\,=\,0\,$ if and only if $x\,=\,\theta$;\\
\textbf{(IFP.11)} $\,\forall t\,\in \,\mathbb{R}^+,\,\nu\,(\,cx\,,\,t\,)\leq\,\nu\,(\,x\,,\,t\,)$ if $\mid c \mid\,\leq\,1,\,\forall\,c\,\in \mathbb{K}$;\\
\textbf{(IFP.12)} $\,\nu\,(\,x\,+\,y\,,\,s\,+\,t\,)\,\leq\,\max\,\{\,\nu\,(\,x\,,\,s\,)\,,\,\nu\,(\,y\,,\,t\,)\,\},\,\,\forall s,\,t\,\in\,\mathbb{R}^+$;\\ 
\textbf{(IFP.13)}  $\mathop{\lim }\limits_{t\, \to\,\,\infty }\,\nu\,(\,x\,,\,t\,)=0.$\\
\textbf{(IFP.14)} if there exists $ \alpha\in(0,1)\,$ such that $\,\nu(x\,,\,t)<\,\alpha\,,\,\forall\,t\,\in\,\mathbb{R}^+\,$ then $x=\theta$.\\
\textbf{(IFP.15)} $\,\forall\,x\,\in\,X\,,\,\, \nu(\,x\,,\,\cdot)\,$ is left continuous on $\,\mathbb{R}\,$.
\end{definition}
Here $(X,\mu,\nu)$ is called intuitionistic fuzzy pseudo normed linear space.

\begin{note}\label{pr1}
\cite{gupta} $a\ast\,a=a$ and $a\diamond\,a=a,\,\forall\,a\in\,[0,1]$ is satisfied only when $a\ast\,b=\min\{a,b\}$ and $a\diamond\,b=\max\{a,b\}$.
\end{note}

\begin{definition}\cite{dinda1}
Let $(X,\mu,\nu)$ be an intuitionistic fuzzy pseudo norm linear space. A sequence $\{a_n\}_{n\in\mathbb{N}}$ converges to $a\in\,X$ if and only if $\mathop {\lim }\limits_{n \to \infty }\mu(a_n-a,t)=1$ and $\mathop {\lim }\limits_{n \to \infty }\nu(a_n-a,t)=0$.
\end{definition}

\begin{theorem}\cite{dinda1}\label{th1}
Let $(X,\mu,\nu)$ be a intuitionistic fuzzy pseudo normed linear space. Then for any $\alpha\in(0,1)$ the functions $\left\| {\,x\,} \right\|_{\,\alpha }\,,\,\left\| {\,x\,} \right\|_{\,\alpha }^{\,\ast}:X\,\rightarrow [0,\infty)$ defined as 
\[\hspace{2 cm}\left\| {\,x\,} \right\|_{\,\alpha }\,=\,\bigwedge\,\{\,t>0\,:\,\mu\,(\,x,t\,)\geq\,\alpha\}\] is a ascending family of pseudo norm on $X	$.
\[\hspace{2 cm}\left\| {\,x\,} \right\|_{\,\alpha }^{\,\ast}\,=\,\bigwedge\,\{\,t>0\,:\,\nu\,(\,x,t\,)\leq \,\alpha\}\] is a descending family of pseudo norm on $X$.
\end{theorem}

\begin{theorem}\cite{dinda1}
Let $(X,\mu,\nu)$ be an intuitionistic fuzzy pseudo normed linear space and let $\,\mu\,^\prime\,,\,\nu\,^\prime\,:\,X\,\times\,\mathbb{R}\,\rightarrow\,[\,0\,,\,1\,]$ be defined by
\[\hspace{0.25 cm}\mu\,^\prime\,(\,x\,,\,t\,)\,=\begin{cases}\,\bigvee\,\{\,\alpha\in\,(0,1)\,:\,\left\| {\,x\,} \right\|_{\,\alpha }\,\leq\,t\},\hspace{0.25 cm}if\,t>0\\
\,0,\hspace{4.25 cm}\,if\,t\,\leq\,0\end{cases}\]

\[\hspace{0.5 cm}\nu\,^\prime\,(\,x\,,\,t\,)\,=\begin{cases}\,\bigwedge\,\{\,\,\alpha\in(0,1)\,:\,\left\| {\,x\,} \right\|_{\,\alpha }^{\,\ast}\,\leq\,t\},\hspace{0.25 cm}if\,t>0\\
\,1,\hspace{4.25 cm}\,if\,t\,\leq\,0\end{cases}\]
then\\
$(i)\,\, (\mu\,^\prime\,,\,\nu\,^\prime\,)$ is an intuitionistic fuzzy pseudo norm on $X$.\\
$(ii)\,\,\mu\,^\prime\,=\,\mu\,$ and $\nu\,^\prime\,=\,\nu,\,$ where $\left\| {\,\cdot\,} \right\|_{\,\alpha }$  is a ascending family of pseudo norms and $\left\| {\,\cdot\,} \right\|_{\,\alpha }^{\,\ast}$ is a descending family of pseudo norms defined in Theorem \ref{th1}.
\end{theorem}
\section{Intuitionistic Fuzzy continuities of operators}
This section deals with the study of different types of continuities of bounded linear operators and their intra relations in intuitionistic fuzzy pseudo normed linear spaces. 

\begin{definition}
Let $(X,\mu_1, \nu_1),\, (Y,\mu_2, \nu_2)$ be intuitionistic fuzzy pseudo normed linear spaces.  A mapping $T:(X,\mu_1, \nu_1) \to (Y,\mu_2, \nu_2)$ is said to be IFC at $x_0 \in X$ if for any given $\epsilon>0$ and $\alpha\in(0,1)$ there exists $\delta=\delta(\alpha, \epsilon)>0$, $\beta=\beta(\alpha, \epsilon)\in(0,1)$ such that for all $x \in X$, 
\[\mu_1(x-x_0,\delta)>1-\beta\,\,\Rightarrow\, \mu_2(T(x)-T(x_0), \epsilon)>1-\alpha\]
\[\nu_1(x-x_0,\delta)<\beta\,\,\Rightarrow\, \nu_2(T(x)-T(x_0), \epsilon)<\alpha.\]
\end{definition}

\begin{definition}
Let $(X,\mu_1, \nu_1),\, (Y,\mu_2, \nu_2)$ be intuitionistic fuzzy pseudo normed linear spaces. A mapping $T:(X,\mu_1, \nu_1) \to (Y,\mu_2, \nu_2)$ is said to be sequentially IFC at $x_0 \in X$ if for any sequence $\{x_n\}_n,\,x_n\in X$ and $t>0$,
\[\mathop{\lim }\limits_{t\, \to\,\,\infty }\,\mu_1(x_n - x_0,t)=1\,\,\Rightarrow\,\mathop{\lim }\limits_{t\, \to\,\,\infty }\,\mu_2(T(x_n)-T(x_0),t)=1\]
\[\mathop{\lim }\limits_{t\, \to\,\,\infty }\,\nu_1(x_n - x_0,t)=0\,\,\Rightarrow\,\mathop{\lim }\limits_{t\, \to\,\,\infty }\,\nu_2(T(x_n)-T(x_0),t)=0.\]
\end{definition} 

\begin{theorem}
If a linear operator $T:(X,\mu_1, \nu_1) \to (Y,\mu_2, \nu_2)$ is sequentially IFC at a point $a_0 \in X$ then it is sequentially IFC on $ X$, where $(X,\mu_1, \nu_1)$ and  $(Y,\mu_2, \nu_2)$ are intuitionistic fuzzy pseudo normed linear spaces.
\end{theorem}
\begin{proof}
Let $\{a_n\}_n$ be a sequence in $X$ and $a_n \to a,\;(a\in\,X)$. Then $\forall\,t>0,\;\mathop{\lim }\limits_{n\, \to\,\,\infty }\mu_1 (a_n -a,t)=1$ and $\mathop{\lim }\limits_{n\, \to\,\,\infty }\nu_1 (a_n -a,t)=0$. Therefore,\\
$\mathop{\lim }\limits_{n\, \to\,\,\infty }\mu_1 ((a_n -a+a_0)-a_0,t)=1$ and $\mathop{\lim }\limits_{n\, \to\,\,\infty }\nu_1 ((a_n -a+a_0)-a_0,t)=0$.\\
Since $T$ is sequentially IFC at $x_0,\;\forall\,t>0$ we have\\
$\mathop{\lim }\limits_{n\, \to\,\,\infty }\mu_1 (T(a_n -a+a_0)-T(a_0),t)=1$ and $\mathop{\lim }\limits_{n\, \to\,\,\infty }\nu_1 (T(a_n -a+a_0)-T(a_0),t)=0$.\\
$\Rightarrow\,\mathop{\lim }\limits_{n\, \to\,\,\infty }\mu_1 (T(a_n) -T(a)+T(a_0)-T(a_0),t)=1$ and $\mathop{\lim }\limits_{n\, \to\,\,\infty }\nu_1 (T(a_n) -T(a)+T(a_0)-T(a_0),t)=0$, since $T$ is linear.\\
$\Rightarrow\,\mathop{\lim }\limits_{n\, \to\,\,\infty }\mu_1 (T(a_n) -T(a),t)=1$ and $\mathop{\lim }\limits_{n\, \to\,\,\infty }\nu_1 (T(a_n) -T(a),t)=0$. Since $a\in\,X$ is arbitrary, $T$ is sequentially IFC on $ X$.
\end{proof}

\begin{theorem}\label{scc1}
A linear operator $T:(X,\mu_1, \nu_1) \to (Y,\mu_2, \nu_2)$ is sequentially IFC if and only if it is IFC, where $(X,\mu_1, \nu_1)$ and  $(Y,\mu_2, \nu_2)$ are intuitionistic fuzzy pseudo normed linear spaces. 
\end{theorem}
\begin{proof}
First suppose $T$ be IFC at $a_0\in\,X$. Let $\{a_n\}_n$ be a sequence in $X$ converges to $a_0$. Then for any given $\epsilon>0,\,\alpha\in(0,1)$ there exist $\delta=\delta(\epsilon,\alpha)>0$ and $\beta=\beta(\epsilon,\alpha)\in(0,1)$ such that $\forall\,a\in X$,
\[\mu_1(a-a_0,\delta)>1-\beta\;\,\Rightarrow\,\mu_2(T(a)-T(a_0),\epsilon)>1-\alpha\]
\[\nu_1(a-a_0,\delta)<\beta\;\Rightarrow\,\nu_2(T(a)-T(a_0),\epsilon)<\alpha\]
Since $a_n$ converges to $a_0$ there exists $n_0\in\mathbb{N}$ such that $\forall\,n_0\geq\,n$,\\
$\mu_1(a_n-a_0,\delta)>1-\beta,\;\nu_1(a_n-a_0,\delta)<\beta$ and since $T$ is IFC at $a_0\in X$, we have\\
$\mu_2(T(a_n)-T(a_0),\epsilon)>1-\alpha$ and $\nu_2(T(a_n)-T(a_0),\epsilon)<\alpha$\\
$\Rightarrow\,T(a_n)\to T(a_0)$ i.e., $T$ is sequentially IFC at $a_0\in\,X$.\\
Conversely, suppose $T$ be not IFC at $a_0\in\,X$. Then there exist $b\in X$ such that for any given $\epsilon>0,\;\alpha\in(0,1)$ there exist $\delta>0,\;\beta\in(0,1)$,\\
$\mu_1(b-a_0,\delta)>1-\beta\;\\\,\Rightarrow\,\mu_2(T(y)-T(x_0),\epsilon)\leq\,1-\alpha$ and $\nu_1(b-a_0,\delta)<\beta\;\\\,\Rightarrow\,\nu_2(T(y)-T(x_0),\epsilon)\geq\,1-\alpha$.\\
Hence for $\delta=\beta=\frac{1}{n+1}$ there exist $b_n$ for $n=1,2,\cdots,$ such that\\
$\mu_1(b_n-a_0,\delta)=\mu_1(b_n-a_0,\frac{1}{n+1})>1-\frac{1}{n+1}\;\\\,\Rightarrow\,\mu_2(T(b_n)-T(a_0),\epsilon)\leq\,1-\alpha$ and \\$\nu_1(b_n-a_0,\delta)=\nu_1(b_n-a_0,\frac{1}{n+1})<\frac{1}{n+1}\;\\\,\Rightarrow\,\nu_2(T(b_n)-T(a_0),\epsilon)\geq\,\alpha$.\\ Therefore, $\mathop{\lim }\limits_{n\, \to\,\,\infty }\mu_1(b_n-a_0,\delta)=1\,\\\,\Rightarrow\,\mathop{\lim }\limits_{n\, \to\,\,\infty }\mu_2(T(b_n)-T(a_0),\epsilon)\neq\,1$ and $\mathop{\lim }\limits_{n\, \to\,\,\infty }\nu_1(b_n-a_0,\delta)=0\,\\\,\Rightarrow\,\mathop{\lim }\limits_{n\, \to\,\,\infty }\nu_2(T(b_n)-T(a_0),\epsilon)\neq\,0$. Hence $T$ is not sequentially IFC at $a_0$.
\end{proof}

\begin{definition}
Let $(X,\mu_1, \nu_1),\, (Y,\mu_2, \nu_2)$ be intuitionistic fuzzy pseudo normed linear spaces. A mapping $T:(X,\mu_1, \nu_1) \to (Y,\mu_2, \nu_2)$ is said to be strongly IFC at $x_0 \in X$ if for any given $\epsilon>0$ there exists $\delta(\epsilon)>0$ such that for all  $x\in X$,
\[\mu_2(T(x)-T(x_0),\epsilon)\,\geq\,\mu_1(x-x_0,\delta)\,,\,\,\nu_2(T(x)-T(x_0),\epsilon)\,\leq\,\nu_1(x-x_0,\delta).\]
\end{definition}

\begin{theorem}\label{sifc1}
If a linear operator $T:(X,\mu_1, \nu_1) \to (Y,\mu_2, \nu_2)$ is strongly IFC at a point $a_0 \in X$ then it is strongly IFC on $X$, where $(X,\mu_1, \nu_1)$ and  $(Y,\mu_2, \nu_2)$ are intuitionistic fuzzy pseudo normed linear spaces. 
\end{theorem}
\begin{proof}
Since $T$ is strongly IFC at $a_0$, for given $\epsilon>0$ there exists 
$\delta(\epsilon)>0$ such that $\forall\,a\in X$,\\
$\mu_2(T(a)-T(a_0),\epsilon)\geq\,\mu_1(a-a_0,\delta)$ and $\nu_2(T(a)-T(a_0),\epsilon)\leq\,\nu_1(a-a_0,\delta)$.\\
Taking $b\in X$ we have $a+a_0-b\in X$. Therefore replacing $a$ by $a+a_0-b$ we have,\\
$\mu_2(T(a+a_0-b)-T(a_0),\epsilon)\geq\,\mu_1(a+a_0-b-a_0,\delta)$ and $\nu_2(T(a+a_0-b)-T(a_0),\epsilon)\leq\,\nu_1(a+a_0-b-a_0,\delta)$.\\
$\Rightarrow\mu_2(T(a)+T(a_0)-T(b)-T(a_0),\epsilon)\geq\,\mu_1(a-b,\delta)$ and $\nu_2(T(a)+T(a_0)-T(b)-T(a_0),\epsilon)\leq\,\nu_1(a-b,\delta)$.\\
$\Rightarrow\mu_2(T(a)-T(b),\epsilon)\geq\,\mu_1(a-b,\delta)$ and $\nu_2(T(a)-T(b),\epsilon)\leq\,\nu_1(a-b,\delta)$.\\
Hence $T$ is strongly IFC at $b$. Since $b\in X$ is arbitrary, $T$ is strongly IFC on $X$. 
\end{proof}

\begin{definition}
Let $(X,\mu_1, \nu_1),\, (Y,\mu_2, \nu_2)$ be intuitionistic fuzzy pseudo normed linear spaces. A mapping $T:(X,\mu_1, \nu_1) \to (Y,\mu_2, \nu_2)$ is said to be weakly IFC at $x_0 \in X$ if for any given $\epsilon>0$ and $\alpha\in(0,1)$ there exists $\delta=\delta(\alpha, \epsilon)>0$ such that for all $x\in X$,
\[\mu_1(x-x_0,\delta)\geq\, \alpha\,\,\Rightarrow\, \mu_2(T(x)-T(x_0), \epsilon)\geq\alpha\]
\[\nu_1(x-x_0,\delta)\leq \alpha\,\,\Rightarrow\, \nu_2(T(x)-T(x_0), \epsilon)\leq \alpha.\]
\end{definition}

\begin{theorem}\label{wifc2}
If a linear operator $T:(X,\mu_1, \nu_1) \to (Y,\mu_2, \nu_2)$ is weakly IFC at a point $a_0 \in X$ then it is weakly IFC on $ X$, where $(X,\mu_1, \nu_1)$ and  $(Y,\mu_2, \nu_2)$ are intuitionistic fuzzy pseudo normed linear spaces.
\end{theorem}
\begin{proof}
Since $T$ is weakly IFC on $a_0\in\,X$ then for given $\epsilon>0,\,\alpha\in(0,1)$ there exist $\delta(\alpha,\epsilon)>0$ such that $\forall\,a\in X$,\\
$\mu_1(a-a_0,\delta)\geq\,\alpha\;\Rightarrow\mu_2(T(a)-T(a_0),\epsilon)\geq\,\alpha$ and $\nu_1(a-a_0,\delta)\leq\,\alpha\;\Rightarrow\nu_2(T(a)-T(a_0),\epsilon)\leq\,\alpha$.\\ 
Taking $b\in X$ we have $a+a_0-b\in X$. Therefore replacing $a$ by $a+a_0-b$ we have,\\
$\mu_1(a-b,\delta)\geq\,\alpha\;\Rightarrow\,\mu_2(T(a+a_0-b)-T(a_0),\epsilon)\geq\,\alpha\;\Rightarrow\,\mu_2(T(a)+T(a_0)-T(b)-T(a_0),\epsilon)\geq\,\alpha\;\Rightarrow\,\mu_2(T(a)-T(b),\epsilon)\geq\,\alpha$ and\\ $\nu_1(a-b,\delta)\leq\,\alpha\;\Rightarrow\,\nu_2(T(a+a_0-b)-T(a_0),\epsilon)\leq\,\alpha\;\Rightarrow\,\nu_2(T(a)+T(a_0)-T(b)-T(a_0),\epsilon)\leq\,\alpha\;\Rightarrow\,\nu_2(T(a)-T(b),\epsilon)\leq\,\alpha$.\\
Since $y(\in\,X)$ is arbitrary, $T$ is weakly IFC on $X$.
\end{proof}  

\begin{theorem}\label{sw}
If a linear operator $T:(X,\mu_1, \nu_1) \to (Y,\mu_2, \nu_2)$ is strongly IFC then it is weakly IFC, where $(X,\mu_1, \nu_1)$ and  $(Y,\mu_2, \nu_2)$ are intuitionistic fuzzy pseudo normed linear spaces.
\end{theorem} 
\begin{proof}
From the definitions of strongly IFC and weakly IFC it follows.
\end{proof}
The next example shows that in an intuitionistic fuzzy pseudo normed linear space weakly intuitionistic fuzzy continuity may not imply strongly intuitionistic fuzzy continuity.  

\begin{example}\label{e1}
Let $(X,\|\cdot\|)$ be a pseudo normed linear space and $\mu,\nu:X\times \mathbb{R}\rightarrow [0,1]\,$ be defined by:
\[\mu(x,t)\,=\begin{cases}\;1\,\,\hspace{2.2 cm}if\;t>0,\,\|\,x\,\|<t.\hspace{0 cm}
\\ \frac{t}{t+\|x\|}\,\,\hspace{1.8 cm}if\,t>0,\,\|\,x\|\geq\,t\,.\,\hspace{-2 cm}
\\0\,\,\,\hspace{2.2 cm}if\,\,t\,\leq\,0.\,\hspace{0 cm}\end{cases}\]
\[\nu(x,t)\,=\begin{cases}0\,\,\,\hspace{2.2 cm}if\;t>0,\,\|\,x\,\|<t.\hspace{0 cm}
\\ \frac{\|x\|}{t+\|x\|}\,\,\,\hspace{1.7 cm}if\,t>0,\,\|\,x\|\geq\,t\,.\,\hspace{-2 cm}
\\1\,\,\,\hspace{2.2 cm}if\,\,t\,\leq\,0.\,\hspace{0 cm}\end{cases}\]
then by Example 3.2 of \cite{dinda1}, $\;(X,\mu,\nu)$ is an IFPNLS. \\
Let $T:(X,\mu,\nu)\to(X,\mu,\nu)$ be a linear operator defined by $T(x)=\frac{x^3}{1+x}$.
\\
Let $x_0\in\,X$ then for each $x\in X,\; \epsilon>0$ and $\alpha\in(0,1)$,\\
$\mu(T(x)-T(x_0),\epsilon)\geq\,\alpha\;\\\,
\Leftarrow\,\dfrac{\epsilon}{\epsilon+\left\|T(x)-T(x_0)\right\|}\geq\,\alpha\,\\\;
\Leftarrow\,\dfrac{\epsilon}{\epsilon+\|\frac{x^3}{1+x}-\frac{{x_0}^3}{1+x_0}\|}\geq\alpha\\
\Leftarrow\,\dfrac{\epsilon\|(1+x)(1+x_0)\|}{\epsilon\|(1+x)(1+x_0)\|+\|x^3 +x_0 x^3 - {x_0}^3-x{x_0}^3\|}\geq\alpha\,\\
\Leftarrow\dfrac{\epsilon\|1+x+x_0+xx_0)}{\epsilon\|1+x+x_0+xx_0)\|+\|(x-x_0)(x^2+xx_0+x^2_0)+xx_0(x+x_0)(x-x_0)\|}\geq\alpha\\
\Leftarrow\dfrac{\epsilon\|1+x+x_0+xx_0)\|}{\epsilon\|1+x+x_0+xx_0)\|+\|(x-x_0)\|\|(x^2+xx_0+x^2_0+x^2x_0+xx^2_0)\|}\geq\alpha\\
\Leftarrow\dfrac{\epsilon\frac{\|1+x+x_0+xx_0)\|}{\|(x^2+xx_0+x^2_0+x^2 x_0+xx^2_0)\|}}{\epsilon\frac{\|1+x+x_0+xx_0)\|}{\|(x^2+xx_0+x^2_0+x^2 x_0+xx^2_0)\|}+\|(x-x_0)\|}\geq\alpha\\
\Leftarrow\,\epsilon\dfrac{\|1+x+x_0+xx_0)\|}{\|(x^2+xx_0+x^2_0+x^2 x_0+xx^2_0)\|}\geq\alpha.\epsilon\frac{\|1+x+x_0+xx_0)\|}{\|(x^2+xx_0+x^2_0+x^2 x_0+xx^2_0)\|}+\alpha\|(x-x_0)\|\\
\Leftarrow\,\epsilon\geq\,\alpha.\epsilon+\alpha\|(x-x_0)\|.\frac{\|(x^2+xx_0+x^2_0+x^2 x_0+xx^2_0)\|}{\|1+x+x_0+xx_0)\|}\,\geq\,\alpha.\epsilon+\alpha\|(x-x_0)\|$, \\since $\dfrac{\|(x^2+xx_0+x^2_0+x^2 x_0+xx^2_0)\|}{\|1+x+x_0+xx_0)\|}\geq\,1$.\\
$\Leftarrow\,\delta\,\geq\,\alpha.\delta+\alpha\|(x-x_0)\|$, by taking $\epsilon=\delta$.\\
$\Leftarrow\,\dfrac{\delta}{\delta+\|(x-x_0)\|}\geq\,\alpha\,\\
\Leftarrow\,\mu(x-x_0,\delta)\geq\,\alpha$; and\\
$\nu(T(x)-T(x_0),\epsilon)\leq\,\alpha\;\\\,\Leftarrow\,\dfrac{\left\|T(x)-T(x_0)\right\|}{\epsilon+\left\|T(x)-T(x_0)\right\|}\leq\,\alpha\,\\\;\Leftarrow\,\dfrac{\|\frac{x^3}{1+x}-\frac{{x_0}^3}{1+x_0}\|}{\epsilon+\|\frac{x^3}{1+x}-\frac{{x_0}^3}{1+x_0}\|}\leq\,\alpha\\\Leftarrow\,\dfrac{\|x^3 +x_0 x^3 - {x_0}^3-x{x_0}^3\|}{\epsilon\|(1+x)(1+x_0)\|+\|x^3 +x_0 x^3 - {x_0}^3-x{x_0}^3\|}\leq\,\alpha\\\,
\Leftarrow\,\dfrac{\|(x-x_0)(x^2+xx_0+x^2_0)+xx_0(x+x_0)(x-x_0)\|}{\epsilon\|1+x+x_0+xx_0)\|+\|(x-x_0)(x^2+xx_0+x^2_0)+xx_0(x+x_0)(x-x_0)\|}\leq\alpha\\
\Leftarrow\,\dfrac{\|(x-x_0)(x^2+xx_0+x^2_0+x^2 x_0+xx^2_0)\|}{\epsilon\|1+x+x_0+xx_0)\|+\|(x-x_0)(x^2+xx_0+x^2_0+x^2x_0+xx^2_0)\|}\leq\alpha\\
\Leftarrow\,\dfrac{\|(x-x_0)\|\|(x^2+xx_0+x^2_0+x^2 x_0+xx^2_0)\|}{\epsilon\|1+x+x_0+xx_0)\|+\|(x-x_0)\|\|(x^2+xx_0+x^2_0+x^2 x_0+xx^2_0)\|}\leq\alpha\\
\Leftarrow\,\dfrac{\|(x-x_0)\|}{\epsilon\frac{\|1+x+x_0+xx_0)\|}{\|(x^2+xx_0+x^2_0+x^2 x_0+xx^2_0)\|}+\|(x-x_0)\|}\leq\alpha\\
\Leftarrow\,\alpha\,\|(x-x_0)\|+\alpha.\epsilon\dfrac{\|1+x+x_0+xx_0)\|}{\|(x^2+xx_0+x^2_0+x^2 x_0+xx^2_0)\|}\geq\,\|(x-x_0)\|\\
\Leftarrow\,(1-\alpha)\|(x-x_0)\|\,\leq\,\alpha.\epsilon\frac{\|1+x+x_0+xx_0)\|}{\|(x^2+xx_0+x^2_0+x^2 x_0+xx^2_0)\|}\leq\,\alpha.\epsilon$,\\ since $\dfrac{\|1+x+x_0+xx_0)\|}{\|(x^2+xx_0+x^2_0+x^2 x_0+xx^2_0)\|}\leq\,1$.\\
$\Leftarrow\,\|(x-x_0)\|-\alpha\|(x-x_0)\|\leq\,\alpha.\delta$, by taking $\delta=\epsilon$\\
$\Leftarrow\,\|(x-x_0)\|\leq\,\alpha(\delta+\|(x-x_0)\|)$\\
$\Leftarrow\,\dfrac{\|(x-x_0)\|}{\delta+\|(x-x_0)\|}\leq\,\alpha$\\
$\Leftarrow\,\nu(x-x_0,\delta)\leq\,\alpha$.\\
Thus for every $\epsilon>0$ and $\alpha\in(0,1)$ there exists $\delta=\delta(\alpha, \epsilon)>0$ such that for all $x\in X$ and $x_0\in\,X$
$\mu(x-x_0,\delta)\geq\, \alpha\,\,\Rightarrow\, \mu(T(x)-T(x_0), \epsilon)\geq\alpha,\;\;\nu(x-x_0,\delta)\leq \alpha\,\,\Rightarrow\, \nu(T(x)-T(x_0), \epsilon)\leq \alpha.$\\
Hence T is weakly intuitionistic fuzzy continuous at $x_0\in X$ and hence on $X$.\\
To show $T$ is not strongly intuitionistic fuzzy continuous, it is enough to for any given $\epsilon>0$ there does not exist a $\delta>0$ such that\\
$\mu(T(x)-T(x_0),\epsilon)\,\geq\,\mu(x-x_0,\delta)\;$ or $\;\nu(T(x)-T(x_0),\epsilon)\,\leq\,\nu(x-x_0,\delta).$\\
Let $\epsilon>0$, then $\forall\,x\in X$ and $x_0\in\,X$,\\ $\mu(T(x)-T(x_0),\epsilon)\,\geq\,\mu(x-x_0,\delta)\\
\Rightarrow\,\mu(\dfrac{x^3}{1+x}-\frac{x^3_0}{1+x_0},\epsilon)\,\geq\,\frac{\delta}{\delta+\|x-x_0\|}$\\
$\Rightarrow\,\dfrac{\epsilon}{\epsilon+\frac{\|x^3+x^3x_0-x^3_0-xx^3_0\|}{\|(1+x)(1+x_0)\|}}\geq\,\frac{\delta}{\delta+\|x-x_0\|}$\\
$\Rightarrow\,\dfrac{\epsilon\|(1+x+x_0+xx_0)\|}{\epsilon\|(1+x+x_0+xx_0)\|+\|x-x_0\|\|x^2+xx_0+x^2_0+x^2x_0+xx^2_0\|}\geq\,\frac{\delta}{\delta+\|x-x_0\|}$\\
$\Rightarrow\,\epsilon\|x-x_0\|\|(1+x+x_0+xx_0)\|\geq\,\delta\|x-x_0\|\|x^2+xx_0+x^2_0+x^2x_0+xx^2_0\|$\\
$\Rightarrow\,\delta\leq\,\epsilon\dfrac{\|(1+x+x_0+xx_0)\|}{\|x^2+xx_0+x^2_0+x^2x_0+xx^2_0\|}$.\\
Now $\inf\{\dfrac{\|(1+x+x_0+xx_0)\|}{\|x^2+xx_0+x^2_0+x^2x_0+xx^2_0\|}\}=0,\;\;\forall\,x\in\,X$. \\
Therefore, $\delta=0$, which is not possible. This shows that $T$ is not strongly intuitionistic fuzzy continuous. 
\end{example}

\begin{theorem}\label{ss}
If a linear operator $T:(X,\mu_1, \nu_1) \to (Y,\mu_2, \nu_2)$ is strongly IFC then it is sequentially IFC, where $(X,\mu_1, \nu_1)$ and  $(Y,\mu_2, \nu_2)$ are intuitionistic fuzzy pseudo normed linear spaces.
\end{theorem} 
\begin{proof}
Let $\{a_n\}_n$ be a sequence in $X$ such that $a_n\to\,a_0$. \\i.e., $\mathop{\lim }\limits_{n\, \to\,\,\infty }\mu_1(a_n-a_0)=1$ and  $\mathop{\lim }\limits_{n\, \to\,\,\infty }\nu_1(a_n-a_0)=0,\;\forall\,t>0$.\\
Now since $T$ is strongly IFC at $a_0\in X$. Then for $\epsilon>0,\;\exists\,\delta(\epsilon)>0$ such that $\forall\,a\in X,$
$\mu_2(T(a)-T(a_0),\epsilon)\geq\,\mu_1(a-a_0,\delta)$  and  $\nu_2(T(a)-T(a_0),\epsilon)\leq\,\nu_1(a-a_0,\delta)$.\\
Now, $\mathop{\lim }\limits_{n\, \to\,\,\infty }\mu_2(T(a_n)-T(a_0),\epsilon)\geq\,\mathop{\lim }\limits_{n\, \to\,\,\infty } \mu_1(a_n-a_0,\delta)=1\,\\\Rightarrow\,\mathop{\lim }\limits_{n\, \to\,\,\infty }\mu_2(T(a_n)-T(a_0),\epsilon)=1$, and\\
$\mathop{\lim }\limits_{n\, \to\,\,\infty }\nu_2(T(a_n)-T(a_0),\epsilon)\leq\,\mathop{\lim }\limits_{n\, \to\,\,\infty } \nu_1(a_n-a_0,\delta)=0\,\\\Rightarrow\,\mathop{\lim }\limits_{n\, \to\,\,\infty }\nu_2(T(a_n)-T(a_0),\epsilon)=0$.\\
Since $\epsilon$ is arbitrary small positive number, $T$ is sequentially IFC.
\end{proof}

The next example shows that in an intuitionistic fuzzy pseudo normed linear space sequentially intuitionistic fuzzy continuity may not imply strongly intuitionistic fuzzy continuity.
\begin{example}
Consider the intuitionistic fuzzy pseudo normed linear space $(X,\mu,\nu)$ as Example \ref{e1} and the linear operator $T$ is defined by $T(x)=\frac{x^3}{1+x}$.\\
Let $\{x_n\}_n$ be a sequence in $X$ such that $x_n\to\,x_0$ in $X$. Now $\forall\,t>0$,\\
$\mathop{\lim }\limits_{n\,\to\,\infty}\mu(x_n-x_0,t)=1\,,\;\;\mathop{\lim }\limits_{n\,\to\,\infty}\nu(x_n-x_0,t)=0$. \\
$\Rightarrow\,\mathop{\lim }\limits_{n\,\to\,\infty}\frac{t}{t+\|x_n-x_0\|}=1\,,\;\;\mathop{\lim }\limits_{n\,\to\,\infty}\frac{\|x_n-x_0\|}{t+\|x_n-x_0\|}=0$.
\begin{equation}\label{eq31}
\Rightarrow\,\mathop{\lim }\limits_{n\,\to\,\infty}\|x_n-x_0\|=0
\end{equation}
Now $\mu(T(x_n)-T(x_0),t)=\dfrac{t}{t+\|\frac{x^3_n}{1+x_n}-\frac{x^3_0}{1+x_0}\|}=\dfrac{t}{t+\|\frac{x^3_n+x_0x^3_n-x^3_0-x_n x^3_0}{(1+x_n)(1+x_0)}\|}\\=\dfrac{t\|(1+x_n)(1+x_0)\|}{t\|(1+x_n)(1+x_0)\|+\|x_n-x_0\|\|x^2_n+x_nx_0+x^2_0+x_0x^2_n+x_nx^2_0\|}=1$ as $n\to \infty$ by Equation \ref{eq31}. Also,\\
$\mu(T(x_n)-T(x_0),t)=\dfrac{\frac{x^3_n}{1+x_n}-\frac{x^3_0}{1+x_0}}{t+\|\frac{x^3_n}{1+x_n}-\frac{x^3_0}{1+x_0}\|}=\dfrac{\|\frac{x^3_n+x_0x^3_n-x^3_0-x_n x^3_0}{(1+x_n)(1+x_0)}\|}{t+\|\frac{x^3_n+x_0x^3_n-x^3_0-x_n x^3_0}{(1+x_n)(1+x_0)}\|}\\=\dfrac{\|x_n-x_0\|\|x^2_n+x_nx_0+x^2_0+x_0x^2_n+x_nx^2_0\|}{t\|(1+x_n)(1+x_0)\|+\|x_n-x_0\|\|x^2_n+x_nx_0+x^2_0+x_0x^2_n+x_nx^2_0\|}=0$ as $n\to \infty$ by Equation \ref{eq31}.\\
Thus $T$ is sequentially IFC at $x_0\in X$ and hence on $X$.
From Example \ref{e1} it is apparent that $T$ is not strongly IFC. 
\end{example}

\begin{corollary}\label{scc2}
If a linear operator $T:(X,\mu_1, \nu_1) \to (Y,\mu_2, \nu_2)$ is strongly IFC then it is IFC, where $(X,\mu_1, \nu_1)$ and  $(Y,\mu_2, \nu_2)$ are intuitionistic fuzzy pseudo normed linear spaces.
\end{corollary}
\begin{proof}
From Theorem \ref{scc1} and Theorem \ref{ss} the corollary follows.
\end{proof}
   

\section{Intuitionistic Fuzzy boundedness of operators}
\begin{definition}
Let $(X,\mu_1, \nu_1),\, (Y,\mu_2, \nu_2)$ be intuitionistic fuzzy pseudo normed linear spaces. A linear operator $T:(X,\mu_1, \nu_1) \to (Y,\mu_2, \nu_2)$ is said to be strongly IFB if $\forall x\in X$ and $\forall\,t\in\,\mathbb{R}^+$,
\[\mu_2(T(x),t)\,\geq\,\mu_1(x,t)\,,\,\,\nu_2(T(x),t)\,\leq\,\nu_1(x,t).\]
\end{definition}

\begin{definition}
Let $(X,\mu_1, \nu_1),\, (Y,\mu_2, \nu_2)$ be intuitionistic fuzzy pseudo normed linear spaces. A mapping $T:(X,\mu_1, \nu_1) \to (Y,\mu_2, \nu_2)$ is said to be weakly IFB if for any $\alpha\in(0,1)$, $\forall x\in X$ and $\forall\,t\in\,\mathbb{R}^+$, 
\[\mu_1(x,t)\,\geq\,\alpha\,\,\Rightarrow\,\mu_2(T(x),t)\geq\,\alpha\,,\,\,\nu_1(x,t)\leq\,1-\alpha \Rightarrow\,\nu_2(T(x),t)\leq\,1-\alpha.\]
\end{definition}

\begin{theorem}\label{sbwb}
If a linear operator $T:(X,\mu_1, \nu_1) \to (Y,\mu_2, \nu_2)$ is strongly IFB then it is weakly IFB, where $(X,\mu_1, \nu_1)$ and  $(Y,\mu_2, \nu_2)$ are intuitionistic fuzzy pseudo normed linear spaces.
\end{theorem} 
\begin{proof}
This theorem easily perceived from the definition of strongly IFB and weakly IFB of linear operators.
\end{proof}


\begin{definition}
Let $(X,\mu_1, \nu_1),\, (Y,\mu_2, \nu_2)$ be intuitionistic fuzzy pseudo normed linear spaces. A mapping $T:(X,\mu_1, \nu_1) \to (Y,\mu_2, \nu_2)$ is said to be uniformly IFB if there exist $c>0,\,\alpha\in (0,1)$ such that 
\[\|T(x)\|_\alpha^2\,\leq\,\left\| {\,x\,} \right\|_{\,\alpha }^1\,,\,\,\,\left\| {\,T(x)\,} \right\|_{\,\alpha }^{\,2\ast}\,\leq\,\left\| {\,x\,} \right\|_{\,\alpha }^{\,1\ast}.\]
Where $\left\| {\,\cdot\,} \right\|_{\,\alpha }^1$\;,  $\;\left\| {\,\cdot\,} \right\|_{\,\alpha }^2$  are ascending family of pseudo norms and $\left\| {\,\cdot\,} \right\|_{\,\alpha }^{\,1\ast}$\;, $\;\left\| {\,\cdot\,} \right\|_{\,\alpha }^{\,2\ast}$ descending family of pseudo norms defined by
\[\left\| {\,x\,} \right\|_{\,\alpha }^1\,=\,\bigwedge\,\{\,t>0\,:\,\mu_1\,(\,x,t\,)\geq\,\alpha\}\;,\;\;\left\| {\,T(x)\,} \right\|_{\,\alpha }^2\,=\,\bigwedge\,\{\,t>0\,:\,\mu_2\,(\,T(x),t\,)\geq\,\alpha\}\]
\[\left\| {\,x\,} \right\|_{\,\alpha }^{\,1\ast}\,=\,\bigwedge\,\{\,t>0\,:\,\nu_1\,(\,x,t\,)\leq \,\alpha\}\;,\;\left\| {\,T(x)\,} \right\|_{\,\alpha }^{\,2\ast}\,=\,\bigwedge\,\{\,t>0\,:\,\nu_2\,(\,T(x),t\,)\leq \,\alpha\}\]
\end{definition}

\begin{theorem}
A linear operator $T:(X,\mu_1, \nu_1) \to (Y,\mu_2, \nu_2)$ is strongly IFB if and only if it is uniformly IFB with respect to corresponding $\alpha$-norms, $\alpha\in(0,1)$, where $(X,\mu_1, \nu_1)$ and  $(Y,\mu_2, \nu_2)$ are intuitionistic fuzzy pseudo normed linear spaces.
\end{theorem} 
\begin{proof}
First we suppose $T$ is strongly IFB. Then $\forall\,x\in X$ and $\forall\,t\in\,\mathbb{R}^+$,
\begin{equation}\label{eq41}
\mu_2(T(x),t)\geq\,\mu_1(x,t)\;,\;\;\nu_2(T(x),t)\leq\,\nu_1(x,t).
\end{equation}
Let $\left\| {\,x\,} \right\|_{\,\alpha }^1\,<t\;\Rightarrow\,\bigwedge\,\{\,s(>0)\,:\,\mu_1\,(\,x,s\,)\geq\,\alpha\}<t.\\
\hspace{2 cm}\Rightarrow\,\exists\,s_0>t$ such that $\mu_1(x,s_0)\geq\,\alpha.\\
\Rightarrow\,\exists\,s_0>t$ such that $\mu_2(T(x),s_0)\geq\,\alpha.$ (by Equation \ref{eq41})\\$\Rightarrow\,\left\| {\,T(x)\,} \right\|_{\,\alpha }^2\,\leq\,s_0<t$.\\
Thus, $\left\| {\,T(x)\,} \right\|_{\,\alpha }^2\,\leq\,\left\| {\,x\,} \right\|_{\,\alpha }^1.$\\
Also, let $\left\| {\,x\,} \right\|_{\,\alpha }^{1\ast}\,>t\;\Rightarrow\,\bigwedge\,\{\,s(>0)\,:\,\nu_1\,(\,x,s\,)\leq\,\alpha\}>t.\\
\Rightarrow\,\exists\,s_0>t$ such that $\nu_1(x,s_0)\leq\,\alpha.\\
\Rightarrow\,\exists\,s_0>t$ such that $\nu_2(T(x),s_0)\leq\,\alpha.$ (by Equation \ref{eq41})\\$\Rightarrow\,\left\| {\,T(x)\,} \right\|_{\,\alpha }^2\,\geq\,s_0>t$.\\
Thus, $\left\| {\,T(x)\,} \right\|_{\,\alpha }^{2\ast}\,\geq\,\left\| {\,x\,} \right\|_{\,\alpha }^{1\ast}.$\\
Hence $T$ is uniformly IFB.\\
Conversely, suppose $T$ is uniformly IFB with respect to corresponding $\alpha$-norms. Then for $\alpha\in\,(0,1)$\\
\begin{equation}\label{eq42}
\left\| {\,T(x)\,} \right\|_{\,\alpha }^2\,\leq\,\left\| {\,x\,} \right\|_{\,\alpha }^1\;\;,\;\;\;\left\| {\,T(x)\,} \right\|_{\,\alpha }^{2\ast}\,\leq\,\left\| {\,x\,} \right\|_{\,\alpha }^{1\ast}
\end{equation}
Let  $\mu_1(x,t)>a\;\,\Rightarrow\,\bigvee\{\alpha\in(0,1):\left\| {\,x\,} \right\|_{\,\alpha }^1\leq\,t\}>a.\\ \Rightarrow\;\exists\,\alpha_0\in(0,1)$  such that $\alpha_0>a$ and $\left\| {\,x\,} \right\|_{\,\alpha_0}^1\leq\,t$\\
$\Rightarrow\;\exists\,\alpha_0\in(0,1)$  such that $\alpha_0>a$ and $\left\| {\,T(x)\,} \right\|_{\,\alpha_0}^2\leq\,t$ (by Equation \ref{eq42})\\
$\Rightarrow\,\mu_2(T(x),t)\geq\,\alpha_0>a$\\
Therefore, $\mu_2(T(x),t)\geq\,\mu_1(x,t)$.\\
Also, let  $\nu_1(x,t)<b\;\,\Rightarrow\,\bigwedge\{\alpha\in(0,1):\left\| {\,x\,} \right\|_{\,\alpha }^{1\ast}\leq\,t\}<b.\\ 
\Rightarrow\;\exists\,\alpha_0\in(0,1)$  such that $\alpha_0<b$ and $\left\| {\,x\,} \right\|_{\,\alpha_0}^{1\ast}\leq\,t$\\
$\Rightarrow\;\exists\,\alpha_0\in(0,1)$  such that $\alpha_0<b$ and $\left\| {\,T(x)\,} \right\|_{\,\alpha_0}^{2\ast}\;\leq\,t$ (by Equation \ref{eq42})\\
$\Rightarrow\,\nu_2(T(x),t)\leq\,\alpha_0<b$\\
Therefore, $\nu_2(T(x),t)\leq\,\nu_1(x,t)$.\\
Hence $T$ is strongly IFB.
\end{proof}

\begin{theorem}\label{cb1}
A linear operator $T:(X,\mu_1, \nu_1) \to (Y,\mu_2, \nu_2)$ is strongly IFC if and only if it is strongly IFB, where $(X,\mu_1, \nu_1)$ and  $(Y,\mu_2, \nu_2)$ are intuitionistic fuzzy pseudo normed linear spaces.
\end{theorem} 
\begin{proof}
First suppose $T$ is strongly IFB then $\forall\,x\in\,X$ and $\forall\,\epsilon\in\,\mathbb{R}^+$,\\
$\mu_2(T(x),\epsilon)\geq\,\mu_1(x,\epsilon)\;,\;\;\nu_2(T(x),\epsilon)\leq\,\nu_1(x,\epsilon)$.\\
$\Rightarrow\,\mu_2(T(x-\theta),\epsilon)\geq\,\mu_1(x-\theta,\epsilon)\;,\;\;\nu_2(T(x-\theta),\epsilon)\leq\,\nu_1(x-\theta,\epsilon)$\\
$\Rightarrow\,\mu_2(T(x)-T(\theta),\epsilon)\geq\,\mu_1(x-\theta,\delta)\;,\;\;\nu_2(T(x-\theta),\epsilon)\leq\,\nu_1(x-\theta,\delta)$, where $\delta=\epsilon$.\\
Therefore $T$ is strongly IFC at $\theta$ and hence by Theorem \ref{sifc1} $T$ is strongly IFC on $X$.\\
Conversely, suppose $T$ is strongly IFC on X. Then $T$ is strongly IFC at any point of $X$, say $\theta$. $\forall\,x\in\,X$ take $\epsilon=t=\delta$, then\\
$\mu_2(T(x)-T(\theta),t)\geq\,\mu_1(x-\theta,t)\;\;,\;\;\nu_2(T(x)-T(\theta),t)\leq\,\nu_1(x-\theta,t)$.\\
$\Rightarrow\,\mu_2(T(x),t)\geq\,\mu_1(x,t)\;\;,\;\;\nu_2(T(x),t)\leq\,\nu_1(x,t)$.\\
If $x=\theta,\;t>0$ then $\mu_2(T(\theta),t)=\mu_2(\theta_Y,t)=1=\mu_1(\theta,t)$ and $\nu_2(T(\theta),t)=\nu_2(\theta_Y,t)=0=\nu_1(\theta,t)$.\\
For any $x,\;t\leq\,0$, $\;\mu_2(T(x),t)=0=\mu_1(x,t)$ and $\nu_2(T(x),t)=1=\nu_1(x,t)$.\\ Hence $\forall\,x\in\,X,\;\forall\,t\in\mathbb{R},\;\;T$ is strongly IFB.
\end{proof}

\begin{corollary}
A linear operator $T:(X,\mu_1, \nu_1) \to (Y,\mu_2, \nu_2)$ is strongly IFB then it is sequentially IFC, where $(X,\mu_1, \nu_1)$ and  $(Y,\mu_2, \nu_2)$ are intuitionistic fuzzy pseudo normed linear spaces.
\end{corollary}
\begin{proof}
From Theorem \ref{ss} and Theorem \ref{cb1} the corollary follows.
\end{proof} 

\begin{corollary}
A linear operator $T:(X,\mu_1, \nu_1) \to (Y,\mu_2, \nu_2)$ is strongly IFB then it is IFC, where $(X,\mu_1, \nu_1)$ and  $(Y,\mu_2, \nu_2)$ are intuitionistic fuzzy pseudo normed linear spaces.
\end{corollary}
\begin{proof}
From Corollary \ref{scc2} and Theorem \ref{cb1} the corollary follows.
\end{proof} 

\begin{theorem}
A linear operator $T:(X,\mu_1, \nu_1) \to (Y,\mu_2, \nu_2)$ is weakly IFC if and only if it is weakly IFB, where $(X,\mu_1, \nu_1)$ and  $(Y,\mu_2, \nu_2)$ are intuitionistic fuzzy pseudo normed linear spaces.
\end{theorem} 
\begin{proof}
First suppose $T$ is weakly IFB. Then for any $\alpha\in(0,1),\;\forall\,x\in\,X,\;\forall\,t\in\mathbb{R}^+$,\\
$\mu_1(x,t)\geq\,\alpha\;\Rightarrow\,\mu_2(T(x),t)\geq\,\alpha$  and $\nu_1(x,t)\leq\,\alpha\;\Rightarrow\,\nu_2(T(x),t)\leq\,\alpha$.\\
$\mu_1(x-\theta,t)\geq\,\alpha\;\Rightarrow\,\mu_2(T(x-\theta),t)\geq\,\alpha$  and $\nu_1(x-\theta,t)\leq\,\alpha\;\Rightarrow\,\nu_2(T(x-\theta),t)\leq\,\alpha$.\\
$\mu_1(x-\theta,\delta)\geq\,\alpha\;\Rightarrow\,\mu_2(T(x)-T(\theta),\epsilon)\geq\,\alpha$  and $\nu_1(x-\theta,\delta)\leq\,\alpha\;\Rightarrow\,\nu_2(T(x)-T(\theta),\epsilon)\leq\,\alpha$, where $\epsilon=t=\delta$. Therefore, $T$ is weakly IFC at $\theta$ and hence by Theorem \ref{wifc2}, $T$ is weakly IFC.\\
Conversely suppose $T$ is weakly IFC on X. Then $T$ is weakly IFC at any point of $X$, say $\theta$. $\forall\,x\in\,X$ take $\epsilon=t=\delta$, then\\
$\mu_1(x-\theta,t)\geq\,\alpha\;\Rightarrow\,\mu_2(T(x)-T(\theta),t)\geq\,\alpha$  and $\nu_1(x-\theta,t)\leq\,\alpha\;\Rightarrow\,\nu_2(T(x)-T(\theta),t)\leq\,\alpha$\\
$\mu_1(x,t)\geq\,\alpha\;\Rightarrow\,\mu_2(T(x),t)\geq\,\alpha$  and $\nu_1(x,t)\leq\,\alpha\;\Rightarrow\,\nu_2(T(x),t)\leq\,\alpha$\\
If $x=\theta,\;t>0$ then $\mu_1(x,t)=1=\mu_2(T(x),t)$ and $\nu_1(x,t)=0=\nu_2(T(x),t)$.\\
For any $x,\;t\leq\,0$, $\;\mu_1(x,t)=0=\mu_2(T(x),t)$ and $\nu_1(x,t)=1=\nu_2(T(x),t)$.\\
Hence for any $\alpha\in(0,1)\;,\;\forall\,x\in\,X,\;\forall\,t\in\mathbb{R},\;\;T$ is weakly IFB.
\end{proof}


\bigskip
\bigskip

{\footnotesize \pn{\bf Bivas Dinda}\; \\ {Department of Mathematics, Mahishamuri Ramkrishna Vidyapith, 
Howrah 711401, West Bengal, India}\\
{\tt Email: bvsdinda@gmail.com}\\

{\footnotesize \pn{\bf Santanu Kumar Ghosh}\; \\ {Department of Mathematics, Kazi Nazrul University, 
Asansol 713340, West Bengal, India}\\
{\tt Email:santanu$_{-}$96@yahoo.co.in}\\

{\footnotesize \pn{\bf T.K. Samanta}\; \\ {Department of Mathematics, Uluberia College, 
Howrah 711315, West Bengal, India}\\
{\tt Email: mumpu$_{-}$tapas5@yahoo.co.in}\\

\begin{thebibliography}{99}
\bibitem{bag1}T. Bag, S.K. Samanta, 
 			{\it Finite dimensional fuzzy normed linear space},  
 			J. Fuzzy Math. \textbf{11}(3) (2003) 687--705 .

\bibitem{bag2}T. Bag, S.K. Samanta, {\it Fuzzy bounded linear operators},  Fuzzy Sets and Systems \textbf{151} (2005) 513--547.

\bibitem{cheng}S.C. Cheng, J.N. Mordeson,  {\it Fuzzy Linear Operators and Fuzzy Normed Linear Spaces}, Bull. Cal. Math. Soc. \textbf{86} (1994) 429--436.

\bibitem{dinda1} B. Dinda, S.K. Ghosh, T.K. Samanta, {\it Intuitionistic fuzzy pseudo normed linear spaces}, New Math. and Nat. Comput. \textbf{15}(1) (2019) 113--127.

\bibitem{dinda2} B. Dinda, T.K. Samanta,  {\it Intuitionistic fuzzy continuity and uniform convergence},J. Open Prob. Compt. Math. \textbf{3}(1) (2010) 8--26.

\bibitem{dinda3}B. Dinda, T.K. Samanta, U.K. Bera,  {\it Gateaux and Fr\' echet derivative in intuitionistic fuzzy normed linear spaces}, New Math. and Nat. Comput. \textbf{8}(3) (2012) 311--322.	

\bibitem{dinda4}B. Dinda, T.K. Samanta, U.K. Bera, {\it Intuitionistic fuzzy Banach algebra}, Bull. of Math. Anal. Appl. \textbf{3}(3) (2011) 273--281.

\bibitem{samanta} B. Dinda, T.K. Samanta, I.H. Jebril, {\it Fuzzy anti-bounded linear operators}, Stud. Univ. Babes-Bolyai Math. \textbf{56}(4) (2011) 123--137.

\bibitem{felbin} C. Felbin,  {\it Finite dimensional fuzzy normed linear spaces}, Fuzzy Sets and Systems \textbf{48}  (1992) 239--248.

\bibitem{golet} I. Golet, {\it On generalized fuzzy normed spaces and coincidence fixed point theorems}, Fuzzy Sets and Systems \textbf{161} (2010) 1138--1144.

\bibitem{gupta}M.M. Gupta, J. Qi,  {\it Theory of T-norms and fuzzy inference method}, Fuzzy Sets and Systems  \textbf{40} (1991) 431--450.

\bibitem{katsaras1}A.K. Katsaras, {\it Fuzzy topological vector spaces I}, Fuzzy Sets and Systems \textbf{6} (1981) 85--95.

\bibitem{katsaras2}A.K. Katsaras, {\it Fuzzy topological vector spaces II}, Fuzzy Sets and Systems \textbf{12} (1984) 143--154.

\bibitem{nadaban}S. N\u ad\u aban, {\it Fuzzy pseudo-norms and fuzzy F-spaces}, Fuzzy Sets and Systems \textbf{282} (2016) 99--114.
 
\bibitem{tvsb1} H.H. Schaefer, M.P. Wolff, {\it Topological Vector Spaces}, Springer, 1999.

\bibitem{saadati}R. Saadati, J.H. Park, {\it On the intuitionistic fuzzy topological spaces}, Chaos Solitons Fractals \textbf{27} (2006) 331--344.

\bibitem{samanta} T.K. Samanta, I.H. Jebril, {\it Finite dimensional intuitionistic fuzzy normed linear spaces}, J. Open Prob. Compt. Math. \textbf{2}(4) (2009) 574--591.

\end{thebibliography}
\end{document}